\documentclass[oneside,english]{amsart}
\usepackage[T1]{fontenc}
\usepackage[latin9]{inputenc}
\usepackage{geometry}
\usepackage{amsthm}
\usepackage{amstext}
\usepackage{amssymb}

\makeatletter
\numberwithin{equation}{section}
\numberwithin{figure}{section}
\theoremstyle{plain}
\newtheorem{thm}{\protect\theoremname}
  \theoremstyle{definition}
  \newtheorem{defn}[thm]{\protect\definitionname}
  \theoremstyle{plain}
  \newtheorem{conjecture}[thm]{\protect\conjecturename}
  \theoremstyle{plain}
  \newtheorem{prop}[thm]{\protect\propositionname}
  \theoremstyle{remark}
  \newtheorem{rem}[thm]{\protect\remarkname}
  \theoremstyle{plain}
  \newtheorem{cor}[thm]{\protect\corollaryname}
  \theoremstyle{plain}
  \newtheorem{lem}[thm]{\protect\lemmaname}

\makeatother

\usepackage{babel}
  \providecommand{\conjecturename}{Conjecture}
  \providecommand{\corollaryname}{Corollary}
  \providecommand{\definitionname}{Definition}
  \providecommand{\lemmaname}{Lemma}
  \providecommand{\propositionname}{Proposition}
  \providecommand{\remarkname}{Remark}
\providecommand{\theoremname}{Theorem}

\newcommand{\mb}{\mathbb}

\newcommand{\rmap}{\rightarrow}

\newcommand{\inj}{\hookrightarrow}

\begin{document}

\title{On the Sato-Tate Conjecture for non-generic Abelian Surfaces}

\author{Christian Johansson \\ \\ with an appendix by Francesc Fit\'e}
\address{School of Mathematics, Institute for Advanced Study, Princeton, NJ 08540, USA}
\email{johansson@math.ias.edu}

\begin{abstract}
We prove many non-generic cases of the Sato-Tate conjecture for abelian
surfaces as formulated by Fit\'e, Kedlaya, Rotger and Sutherland, using
the potential automorphy theorems of Barnet-Lamb, Gee, Geraghty and
Taylor.
\end{abstract}
\maketitle

\section{Introduction}

Let $E$ be an elliptic curve over a number field $F$ without complex
multiplication (CM) and let $v$ be a finite prime of $F$ with residue
field $\mathbb{F}_{q_{v}}$ such that $E$ has good reduction at $v$.
Hasse's Theorem tells us that the number of $\mathbb{F}_{q_{v}}$-points
of $E$ is $q_{v}+1-a_{v}$, where $a_{v}$ is an integer such that
$|a_{v}|\leq2q_{v}^{1/2}$. The Sato-Tate conjecture concerns the
distribution of the numbers $a_{v}/2q_{v}^{1/2}$ in $[-1,1]$ as
$v$ varies through the primes of good reduction for $E$. More precisely,
it tells us that the numbers $\theta_{v}=\arccos\,(a_{v}/2q_{v}^{1/2})$
are distributed according to the measure $\tfrac{2}{\pi}sin^{2}\theta\, d\theta$
on $[0,\pi]$. This conjecture was recently settled in \cite{HSBT}
and \cite{BGHT} (see also \cite{BLGG}) when $F$ is totally real
as a result of progress on potential automorphy theorems for compatible
systems of Galois representation of arbitrary dimension. When $E$
has CM over $F$, the $\theta_{v}$ are uniformly distributed in $[0,\pi]$;
when $E$ has CM but not over $F$ the $\theta_{v}$ are equidistributed on $[0,\pi]$ according to the measure which
is half of the uniform probability measure over $[0,\pi]$ plus a Dirac measure
of mass $1/2$ at $\pi/2$.

All known proofs follow a pattern similar to that of the classical proof of the
Prime Number Theorem and its generalizations, where holomorphy and
non-vanishing of certain (complex) $L$-functions are established
in the region $Re\, s\geq1$. In the CM cases, these $L$-functions
are understood in terms of Hecke $L$-functions. In the non-CM case,
Tate realized that the same proof would work if one could establish
the required properties for the symmetric power $L$-functions of
the elliptic curve. These observations were generalised and axiomatized by Serre (\cite[Appendix
to \S 1]{Ser}). It was then observed by Langlands that the
required analytic properties follow from the automorphy of associated
$L$-functions, and this opened up possibilities both in terms of proofs
and further generalizations. During these developments, the Sato-Tate
conjecture was also reinterpreted as a statement about the equidistribution
of a certain sequence of elements (the images of Frobenii, suitably
normalised) inside a certain compact Lie group (that turns out to
be ${\rm SU}(2)$ for the original Sato-Tate conjecture). If one believes
that all motivic $L$-functions are automorphic, then the Sato-Tate conjecture becomes
a statement about the distribution of the Satake parameters of unitary
automorphic representations. 

In \cite{FKRS}, Fit\'e, Kedlaya, Rotger and Sutherland described the
generalization of the Sato-Tate conjecture to abelian surfaces. Before
we describe their results, let us take a step back and sketch how
a fully functional Langlands philosophy would apply to the situation
of abelian surfaces (the same type of argument should apply to any "pure motive"). Let $A/F$ be an abelian surface. Then the Langlands
correspondence predicts the existence of an $L$-algebraic automorphic
representation $\pi_{A}$ attached to the compatible system $(H_{et}^{1}(A,\overline{\mathbb{Q}}_{\ell}))_{\ell}$
and associated with $\pi_{A}$ is a representation $\rho_{A}\,:\, L_{F}\rightarrow{\rm GL}_{4}(\mathbb{C})$
of the conjectural Langlands group $L_{F}$. We define the Sato-Tate
group of $A$ to be the image $G=(\rho_{A}\otimes|\cdot|^{1/2})(L_{F})$
where $|\cdot|$ denotes the norm character of $L_{F}$. The representation
$\rho_{A}\otimes|\cdot|^{1/2}$ is unitary and the group $G$ should
be a compact group. The Sato-Tate conjecture would then say that the
sequence of conjugacy classes $((\rho_{A}\otimes|\cdot|^{1/2})(Frob_{v}))_{v\notin S}$
(where $S$ is the set of finite primes where $A$ has bad reduction)
is equidistributed with respect to the measure on the set of conjugacy
classes of $G$ induced from the Haar probability measure on $G$. Following the strategy
of Serre, this follows from the holomorphy and non-vanishing in the
region $Re\, s\geq1$ of the partial $L$-functions $L^{S}(s,r\circ(\rho_{A}\otimes|\cdot|^{1/2}))$
for all nontrivial irreducible representations $r$ of $G$. The global
Langlands correspondences would then imply that each $L^{S}(s,r\circ(\rho_{A}\otimes|\cdot|^{1/2}))$
is the partial $L$-function of a unitary cuspidal automorphic representation
of ${\rm GL}_{\dim\, r}(\mathbb{A}_{F})$ (which is not of form $|\cdot|^{i\lambda}$,
$\lambda\in\mathbb{R}$) and this gives the holomorphy and non-vanishing.
All this is of course highly conjectural and the existence of $L_{F}$ is particularly problematic. Instead, in order to get
something well-defined, one has to use the compatible system $(H_{et}^{1}(A,\overline{\mathbb{Q}}_{\ell}))_{\ell}$
directly; this a construction due to Serre (\cite[ \S 8]{Ser2}). This gives a substitute (denoted $ST_{A}$) for $G$ as well as the conjugacy classes
$((\rho_{A}\otimes|\cdot|^{1/2})(Frob_{v}))_{v\notin S}$, and Serre formulates the Sato-Tate conjecture in this generality. For abelian surfaces, Fit\'e, Kedlaya, Rotger and Sutherland (in \cite{FKRS}) give a precise conjecture of what $ST_{A}$ should be; they obtain 52 possibilities, each given by an explicit recipe in terms of the endomorphisms of $A$.

In this paper we wish to investigate what can be said about the Sato-Tate
conjecture for abelian surfaces using current potential automorphy
theorems. Roughly speaking there are three cases. First, there is
the so-called generic case when ${\rm End}_{\overline{\mathbb{Q}}}(A)=\mathbb{Z}$.
In this case we are not able to prove anything because of current
restrictions on potential automorphy theorems to Galois representations
that are regular. Second, we have the cases when $H_{et}^{1}(A,\overline{\mathbb{Q}}_{\ell})$
is potentially abelian where one may prove the conjecture completely,
using class field theory. The remaining cases may be loosely described as those are those that are potentially
of ${\rm GL}_{2}$-type (see Definition \ref{gl2def} for our slightly non-standard definition of ${\rm GL}_{2}$-type) but not potentially abelian; here we may prove
the conjecture under the restriction that a certain (at most quadratic) extension of $F$
is totally real. Unsurprisingly, the situation and required assumptions
mirror that of the elliptic curve case. The proofs are somewhat more
delicate due to the more complicated structure of $ST_{A}$ but otherwise follow
the general strategy as in \cite{HSBT}. The results we prove may
be summarized loosely as follows:
\begin{thm}
\label{thm: main thm}Let $A/F$ be an abelian surface.

1) (Propositions \ref{prop: B[C_1]}, \ref{prop: B[C_2]}, \ref{prop: E[C_n]}, \ref{prop: E[R]} and \ref{prop: E[D_n]}) If $A$ is
of type $\textbf{B}$, $\textbf{C}$ or $\textbf{E}$ and a certain
at most quadratic extension of $F$ is totally real, then the Sato-Tate conjecture
holds for $A$.

2) (Proposition \ref{cor: D or F}) If $A$ is of type $\textbf{D}$
or $\textbf{F}$, then the Sato-Tate conjecture holds for $A$.
\end{thm}

Here type $\textbf{B}$, $\textbf{C}$, $\textbf{D}$, $\textbf{E}$
resp. $\textbf{F}$ refers to the (absolute) type as defined in \cite[\S 4]{FKRS},
we recall them in \S \ref{sec: description}. For now, it suffices
to say that $\textbf{D}$ and $\textbf{F}$ are the potentially abelian
cases, and that the excluded type $\textbf{A}$ is the generic case.
Special cases of Theorem \ref{thm: main thm} that have previously
been recorded in the literature are the case when $A/F$ is isogenous
to a product of non-CM elliptic curves that do not become isogenous
over any finite extension and $F$ is totally real (\cite{Har}), and
a number of cases of type $\textbf{F}$ over $\mathbb{Q}$ (\cite{FS}).

Let us outline the contents of the paper. Section \ref{sec: description}
states the conjecture, explains the strategy and sets up notation
and terminology to be used throughout the paper. The reader is advised to
read this before proceeding, or look back at it later if he encounters unfamiliar
notation. In section \ref{sec: abelian} we do the potentially abelian
cases. Indeed, we give a general equidistribution result for continuous
unitary representations $W_{F}\rightarrow{\rm GL}_{n}(\mathbb{C})$,
roughly following the ``ideal'' strategy above. This is certainly well known but we could not find a precise reference, so we have
included some details. For abelian varieties $A$ whose Tate module is potentially abelian this gives a "Sato-Tate conjecture" for $A$, but it is not clear that the group obtained in this fashion agrees with the one defined by Serre. We also prove the Sato-Tate conjecture in the form of Serre.
In section \ref{sec: GL(2)-type} we do the cases potentially
of ${\rm GL}_{2}$-type, using the powerful potential automorphy results
of \cite{BLGGT}. We believe,
and hope that the reader agrees, that the relatively clean proofs
of \S \ref{sec: GL(2)-type} are a good illustration of the power
and beauty of the results of \cite{BLGGT}.
The paper concludes with an appendix, written by Francesc Fit\'e, giving examples of abelian surfaces satisfying the assumptions of Propositions \ref{prop: B[C_2]} and \ref{prop: E[D_n]}, showing that they are non-empty.

The author wishes to thank Toby Gee for useful remarks on an earlier draft of this paper, Kiran Kedlaya and Jean-Pierre Serre for highlighting the issue of the two different definitions of a Sato-Tate group in the potentially abelian case, Andrew Wiles for a useful conversation, and the anonymous referee for useful comments and corrections. Moreover thanks are also due to Kedlaya, Benjamin Smith and Francesc Fit\'e for discussions relating to the appendix. This research was supported by EPSRC Grant EP/J009458/1.

\section{\label{sec: description}The Conjecture and the Strategy}

We will fix once and for all isomorphism $\iota\,:\,\overline{\mathbb{Q}}_{\ell}\cong\mathbb{C}$ for all $\ell$.
(For simplicity we will omit $\ell$ from the notation; we believe
this should not cause any confusion). We will let $\overline{\mathbb{Q}}$ denote the algebraic closure of $\mathbb{Q}$ in $\mathbb{C}$. If $K$ is any number field then, using $\iota$, we may identify embeddings $K\inj \overline{\mb{Q}}_{\ell}$ and $K\inj \overline{\mb{Q}}$.

\subsection{\label{sec: ST-group}The Sato-Tate group and the definition of conjugacy classes}

In this subsection we recall \cite[\S 2.1]{FKRS}. Let $A$ be
an abelian variety of dimension $g$ over a number field $F$, let $\phi$ be a polarization
of $A$ and fix an embedding $F\hookrightarrow\overline{\mathbb{Q}}$.
Fix a symplectic basis for $H_{1}(A(\mathbb{C})^{an},\mathbb{Q})$
and use it to equip that space with an action of ${\rm GSp}_{2g}(\mathbb{Q})$.
Fix a prime $\ell$ and let $V_{\ell}(A)$ be the rational $\ell$-adic
Tate module of $A$. Make the identifications 
\[
V_{\ell}(A)\cong H_{1,et}(A_{\overline{\mathbb{Q}}_{\ell}},\mathbb{Q}_{\ell})\cong H_{1,et}(A_{\mathbb{C}},\mathbb{Q}_{\ell})\cong H_{1}(A(\mathbb{C})^{an},\mathbb{Q}_{\ell})\cong H_{1}(A(\mathbb{C})^{an},\mathbb{Q})\otimes_{\mathbb{Q}}\mathbb{Q}_{\ell}
\]
The Weil pairing becomes identified with the cup product pairings
in etale and singular cohomology and our symplectic basis for $H_{1}(A(\mathbb{C})^{an},\mathbb{Q})$
gives a symplectic basis for $V_{\ell}(A)$, and the action of $G_{F}=Gal(\overline{\mathbb{Q}}/F)$
defines a continuous homomorphism 
\[
\rho_{A,\ell}\,:\, G_{F}\rightarrow{\rm GSp}_{2g}(\mathbb{Q}_{\ell}).
\]
We let $G_{\ell}=G_{\ell}(A)$ denote the image of $\rho_{A,\ell}$,
and $G_{\ell}^{Zar}=G_{\ell}^{Zar}(A)$ denotes the Zariski closure
of $G_{\ell}$ inside ${\rm GSp}_{2g}(\mathbb{Q}_{\ell})$ (the $\ell$-adic
monodromy group of $A$). $G_{\ell}$ is open in $G_{\ell}^{Zar}$
in the $\ell$-adic topology. Going further we let $G_{F}^{1}$ denote
the kernel of the cyclotomic character $\chi_{\ell}\,:\, G_{F}\rightarrow\mathbb{Z}_{\ell}^{\times}$
and set $G_{\ell}^{1}=G_{\ell}^{1}(A)$ to be the image of $G_{F}^{1}$
under $\rho_{A,\ell}$; we denote by $G_{\ell}^{1,Zar}=G_{\ell}^{1,Zar}(A)$
the Zariski closure of $G_{\ell}^{1,Zar}$ in ${\rm GSp}_{2g}(\mathbb{Q}_{\ell})$.
$G_{\ell}^{1,Zar}$ coincides with the kernel of the similitude character
on $G_{\ell}^{Zar}$. If $F^{\prime}/F$ is a finite extension, then
$G_{\ell}^{1,Zar}(A_{F^{\prime}})$ is a finite index subgroup of
$G_{\ell}^{1,Zar}(A)$, hence these groups have the same identity
components. Moreover, for sufficiently large $F^{\prime}$, $G_{\ell}^{1,Zar}(A_{F^{\prime}})$
is connected.

Using $\iota$ we may embed ${\rm GSp}_{2g}(\mathbb{Q}_{\ell})$ into
${\rm GSp}_{2g}(\mathbb{C})$. Put $G^{1}=G_{\ell}^{1,Zar}\otimes_{\mathbb{Q}_{\ell},\iota}\mathbb{C}$
and $G=G_{\ell}^{Zar}\otimes_{\mathbb{Q}_{\ell},\iota}\mathbb{C}$.
Let $v$ be a finite prime of $F$ with residue field $\mathbb{F}_{q_{v}}$(we
will use this notation throughout the paper). We may identify $G/G^{1}$
with $\mathbb{C}^{\times}$ compatibly with the cyclotomic character,
and the image of $g_{v}=\rho_{A,\ell}(Frob_{v})\in G$ in $\mathbb{C}^{\times}$
is $q_{v}$ (here and throughout the rest of this paper $Frob_{v}$
denotes the arithmetic Frobenius). An argument due to Deligne shows that $G_{\ell}^{Zar}$ contains the central $\mathbb{G}_{m}$ of ${\rm GSp}_{2g}$. Hence $g_{v}^{\prime}=q_{v}^{-1/2}g_{v}\in G^{1}$. 
\begin{defn}
1) The Sato-Tate group $ST_{A}$ of $A$ (for the prime $\ell$ and isomorphism
$\iota$) is a maximal compact Lie subgroup of $G^{1}$ contained
in ${\rm USp}(2g)$.

2) The semisimple component of $g_{v}^{\prime}$ is an element of
$G^{1}$ with eigenvalues of norm $1$ and hence belongs to some conjugate
of $ST_{A}$. Thus we may define $s(v)$ to be the associated conjugacy
class in $ST_{A}$.\end{defn}

\begin{rem}
We have phrased the definitions in the general form of  \cite[\S 8]{Ser2}. In the case of abelian varieties, semisimplicity is well known and due to Tate. Moreover, $G_{\ell}^{Zar}$, and hence $G_{\ell}^{1,Zar}$ and $ST_{A}$, is reductive by Faltings's theorem.
\end{rem}

We remark that when $g\leq 3$ or $A$ is isogenous to a product of CM abelian varieties, the $G_{\ell}^{1,Zar}$ is known to have a common model over $\mathbb{Q}$ (\cite[Theorem 2.16]{FKRS}, \cite[Theorem 6.6]{BK}), so $ST_{A}$ is a compact real form of $G^{1}$ and independent of $\ell$. Note however that it is not in general known in these cases that the conjugacy classes $s_{v}$ are independent of $\ell$. We may now state the generalised Sato-Tate conjecture:  

\begin{conjecture}
(\cite[\S 8]{Ser2}, \cite[Conjecture 1.1]{FKRS}) The classes $s(v)$ are equidistributed
in the set $Conj(ST_{A})$ of conjugacy classes of $ST_{A}$ with
respect to the pullback of the Haar probability measure on $ST_{A}$
to $Conj(ST_{A})$.
\end{conjecture}

\subsection{\label{sec: strategy} The strategy of proof (following Tate and Serre)}

Let $S$ be a finite set of places of $F$ containing the infinite
places such that $\rho_{A,\ell}$ is unramified outside $S$. Let
$r$ be an irreducible representation of $ST_{A}$. We may form the
formal product 
\[
L^{S}(s,r,A)=\prod_{v\notin S}\frac{1}{{\rm det}(1-r(s(v))q_{v}^{-s})}
\]
where by abuse of notation we let $s(v)$ denote some element of the
conjugacy class $s(v)$ (any element will do; ${\rm det}(1-r(s(v))q_{v}^{-s})$
is independent of the choice). Then Serre, elaborating on the special
case $ST_{A}={\rm SU}(2)$ which was studied by Tate, proves the following:
\begin{thm}
(Serre, \cite[Appendix to \S 1]{Ser}) Assume that, for any irreducible
representation $r$, $L^{S}(s,r,A)$ converges absolutely on $Re(s)>1$
and extends to a meromorphic function on $Re(s)\geq1$ having no zeroes
or poles except possibly at $s=1$. Then $(s(v))_{v\notin S}$ are
equidistributed in $Conj(ST_{A})$ if and only if $L^{S}(s,r)$ is holomorphic
and non-vanishing at $s=1$ for all irreducible $r\neq1$.
\end{thm}
In order to prove anything about the functions $L^{S}(s,r,A)$ we will have to identify them with partial
$L$-functions of geometric Galois representations. Since the analytic
properties we require only depend on knowing the Euler factors at
all but finitely many places we will not worry too much about the
distinction between the full $L$-function and its various incomplete/partial
$L$-functions. This should hopefully not cause any confusion. The
basic strategy, due to Taylor (\cite{Tay}, see also \cite{HSBT}),
for proving these sorts of analytic results is to combine potential
automorphy results with Brauer's Theorem and known analytic properties
of various (complex analytic) $L$-functions. The former are currently
restricted to regular Galois representations (i.e. with distinct Hodge-Tate
weights); using the latter techniques one may enlarge the class of $L$-functions
whose analytic behaviour can be studied to handle the cases required
for Theorem \ref{thm: main thm}.

\subsection{The cases}

In \cite{FKRS}, 52 cases of possible Sato-Tate groups are identified
(35 of them possible over totally real fields, 34 over $\mathbb{Q}$),
of which we will exclude the case $ST_{A}={\rm USp}(4)$ (the generic
case) as it seems to be intractable using current technology. The
$52$ cases correspond to $52$ so-called ``Galois types'', giving
the abstract structure of ${\rm End}_{L}(A)\otimes_{\mathbb{Z}}\mathbb{R}$
as a representation of $Gal(L/F)$, where $L$ is the minimal extension
of $F$ over which all endomorphisms of $A$ are defined. After excluding
the generic case (Galois type $\textbf{A}$), the remaining Galois
types are divided into 5 ``types'' (or ``absolute types'') $\textbf{B}$,
$\textbf{C}$, $\textbf{D}$, $\textbf{E}$ and $\textbf{F}$. They
correspond to the following arithmetic interpretations:
\begin{itemize}
\item $\textbf{B}$ : $A_{L}$ is either isogenous to a product of nonisogenous
elliptic curves without CM or simple with real multiplication.
\item $\textbf{C}$ : $A_{L}$ is isogenous to a product of two elliptic
curves, one with CM and the other one without.
\item $\textbf{D}$ : $A_{L}$ is either isogenous to a product of nonisogenous
elliptic curves with CM or simple with CM by a quartic field.
\item $\textbf{E}$ : $A_{L}$ is either isogenous to the square of an elliptic
curve without CM or simple with QM (quaternionic multiplication).
\item $\textbf{F}$ : $A_{L}$ is isogenous to the square of an elliptic
curve with CM.
\end{itemize}
We will obtain complete results for cases $\textbf{D}$ and $\textbf{F}$
from a general study of potentially abelian, geometric Galois representations
in \S \ref{sec: abelian} and do a case-by-case study of types $\textbf{B}$,
$\textbf{C}$ and $\textbf{E}$ in section \ref{sec: GL(2)-type}.

\subsection{Some notation and terminology}

When classifying the irreducible representations of the Sato-Tate
groups we will use specific presentations of them as subgroups of
${\rm USp}(4)$, which we will give as a subgroup of ${\rm GL}_{4}(\mathbb{C})$
(using different alternating forms in different cases). To make this
paper easier to read in conjunction with \cite{FKRS}, we have tried
to stick with their notation (which can be found in \S 3 of their
paper) as much as possible; the main difference is that we use $F$
to denote the field of definition of $A$ (they use $k$) and $L$
to denote the minimal field of definition of all endomorphisms (they
use $K$; we will instead use $K$ to denote a quadratic field inside
${\rm End}_{L}(A)\otimes_{\mathbb{Z}}\mathbb{Q}$). We set ${\rm End}_{F^{\prime}}(A)^{0}={\rm End}_{F^{\prime}}(A)\otimes_{\mathbb{Z}}\mathbb{Q}$
for any extension $F^{\prime}/F$. Next, let us recall the notation
of \cite{FKRS} for various matrices and subgroups of ${\rm GL}_{4}(\mathbb{C})$
here for the convenience of the reader, and make a few additions.
Define 
\[
J_{2}=\left(\begin{array}{cc}
0 & 1\\
-1 & 0
\end{array}\right)\in{\rm GL}_{2}(\mathbb{C}),\qquad J=\left(\begin{array}{cc}
0 & J_{2}\\
-J_{2} & 0
\end{array}\right)\in{\rm GL}_{4}(\mathbb{C})
\]
and 
\[
S=\left(\begin{array}{cc}
0 & Id_{2}\\
-Id_{2} & 0
\end{array}\right),\qquad S^{\prime}=\left(\begin{array}{cc}
J_{2} & 0\\
0 & J_{2}
\end{array}\right).
\]
We use the standard presentations 
\[
{\rm U}(n)=\left\{ A\in{\rm GL}_{n}(\mathbb{C})\mid A^{\ast}A=Id_{n}\right\} 
\]
\[
{\rm SU}(n)=\left\{ A\in{\rm U}(n)\mid det(A)=1\right\} 
\]
for the unitary and special unitary groups, where $A^{\ast}$ denotes
the conjugate-transpose of $A$. For type $\textbf{B}$ and $\textbf{C}$
we will use the presentation 
\[
{\rm USp}(4)=\left\{ A\in{\rm U}(4)\mid A^{t}S^{\prime}A=S^{\prime}\right\} .
\]
For type $\textbf{E}$ we use the presentation 
\[
{\rm USp}(4)=\left\{ A\in{\rm U}(4)\mid A^{t}SA=S\right\} .
\]

We record three matrices here that will be of use in later sections:
\[
a=\left(\begin{array}{cc}
J_{2}\\
 & Id_{2}
\end{array}\right),\quad b=\left(\begin{array}{cc}
Id_{2}\\
 & J_{2}
\end{array}\right),\quad c=\left(\begin{array}{cc}
 & Id_{2}\\
-Id_{2}
\end{array}\right)
\]
Note the slight inconsistency $c=S$; note also that $ab=S^{\prime}$.
This follows the usage in \cite{FKRS} (though they have no notation
for what we are calling $S^{\prime}$ other than $ab$); $S$ and
$S^{\prime}$ will only be used as matrices defining symplectic forms
and $a$, $b$, $c$ and their products will only be used to describe
elements of our Sato-Tate groups.

We will use the word character for what is sometimes called a quasicharacter,
i.e. we do not require that characters are unitary unless specified.
We will follow the usage in \cite{Har} and say that a meromorphic
function $f\,:\, U\rightarrow\mathbb{C}$ where $U\subseteq\mathbb{C}$
is open and contains $\{Re\, s\geq1\}$ is invertible if it is holomorphic
and non-vanishing on $\{Re\, s\geq1\}$. Following the usage in \cite{FKRS}
we will let $Frob_{v}$, for $v$ a finite place of a number field
$F$, denote an arithmetic Frobenius element, and we will use the
Tate module of our abelian surfaces (as opposed to their duals which
we used in the introduction). In accordance with this choice we will
normalise local class field theory by sending uniformizers to arithmetic
Frobenii and attach Galois representation to automorphic representations
of general linear groups by matching up Satake parameters with eigenvalues
of arithmetic Frobenii.

\section{\label{sec: abelian}Equidistribution Laws for Potentially Abelian
geometric Galois Representations}

The aim of this section is to record equidistribution results for
$\ell$-adic Galois representations that are geometric and potentially
abelian. This is done by giving a natural correspondence between such
Galois representation and Weil parameters (i.e. continuous complex
representations of the global Weil group) that is characterized by
the fact that it preserves trace of Frobenius at unramified places.
For unitary Weil parameters satisfying a mild condition on the image
one may prove an equidistribution result. Neither statements nor proofs
in this section should surprise an expert, but we have not found a
reference for some of the results that we present (the abelian case
is done in \cite{Ser}). Some of the results in this section are covered in \cite{Tat} and \cite{Wei}; we have included a few details for the convenience of the reader.

\subsection{Weil parameters and potentially abelian Galois representations}

Let us start by defining the main objects of interest in this section. At the suggestion of the referee, we sketch the definition of the Weil group of a number field $F$, and we refer to \cite{AT} and \cite{Tat} for the details. Given a finite Galois extension $E/F$, the relative Weil group $W_{E/F}$ is the extension
$$ 1 \rightarrow C_{E} \rightarrow W_{E/F} \rightarrow Gal(E/F) \rightarrow 1 $$
defined by the fundamental class $\alpha_{E/F}\in H^{2}(Gal(E/F),C)$ of global class field theory, where $C_{E}=E^{\times}\backslash \mathbb{A}_{E}^{\times}$ is the idele class group and $C= \varinjlim_{E}C_{E}$ (direct limit taken over all finite $E/F$). $W_{E/F}$ inherits a group topology from $C_{E}$ and the Weil group $W_{F}$ is the topological group $W_{F}=\varprojlim_{E}W_{E/F}$. 
\begin{defn}
1) A Weil parameter is a continuous semisimple complex finite-dimensional
representation of $W_{F}$.

2) A Weil parameter $r\,:\, W_{F}\rightarrow{\rm GL}_{n}(\mathbb{C})$
is said to be algebraic if for any place $v\mid\infty$ of $F$, $r|_{W_{\overline{F_{v}}}}$
is a direct sum of characters of the form $z\mapsto z^{p}\bar{z}^{q}$
(for some $p,q\in\mathbb{Z}$, using an identification $W_{\overline{F_{v}}}\cong\mathbb{C}^{\times}$).

3) A semisimple Galois representation $\rho\,:\, G_{F}\rightarrow{\rm GL}_{n}(\overline{\mathbb{Q}}_{\ell})$
is said to be potentially abelian if there exists a finite extension
$E/F$ such that $\rho|_{G_{E}}$ is a direct sum of characters.
\end{defn}
There is a strong link between algebraic Weil parameters and semisimple
geometric potentially abelian $\ell$-adic representations as the
following well-known proposition indicates. Recall that a continuous
irreducible representation of a topological group is called primitive
if it is not induced from any open finite index subgroup.
\begin{prop}
\label{prop: primitive}1) If $r\,:\, W_{F}\rightarrow{\rm GL}_{n}(\mathbb{C})$
is a primitive Weil parameter, then there exists a finite image representation
$\sigma\,:\, G_{F}\rightarrow{\rm GL}_{n}(\mathbb{C})$ and a character
$\chi\,:\, W_{F}\rightarrow\mathbb{C}^{\times}$ such that $r=\sigma\otimes\chi$.

2) If $\rho\,:\, G_{F}\rightarrow{\rm GL}_{n}(\overline{\mathbb{Q}}_{\ell})$
is primitive and potentially abelian, then there is a finite image
representation $\tau\,:\, G_{F}\rightarrow{\rm GL}_{n}(\overline{\mathbb{Q}}_{\ell})$
and a character $\epsilon\,:\, G_{F}\rightarrow\overline{\mathbb{Q}}_{\ell}^{\times}$
such that $\rho=\tau\otimes\epsilon$.\end{prop}
\begin{proof}
1) We give the proof, following \cite[\S 2.2.3]{Tat}.  As $W_{F}=\varprojlim_{E} W_{E/F}$ and
${\rm GL}_{n}(\mathbb{C})$ has no small subgroups we can find a finite Galois
extension $E/F$ such that $r$ factors through $W_{E/F}$.
Then $W_{E}^{ab}$ is an abelian normal subgroup of $W_{E/F}$ so
since $r$ is primitive Clifford's theorem tells us that $W_{E}^{ab}$
acts by scalars. Thus ${\rm proj}\, r$ factors through $Gal(E/F)$.
By Tate's result that $H^{2}(G_{F},\mathbb{C}^{\times})=1$ ${\rm proj}\, r$
lifts to representation $\sigma\,:\, G_{F}\rightarrow{\rm GL}_{n}(\mathbb{C})$
(necessarily of finite image and irreducible). As $r$ and $\sigma$
have the same projectivization, there is a character $\chi$ of $W_{F}$
such that $r=\sigma\otimes\chi$.

2) The proof is the same as 1), with a few minor changes. First, by
assumption we can find $E/F$ finite such that $\rho|_{G_{E}}$ is
abelian. As in 1) Clifford's theorem implies that $\rho|_{G_{E}}$
acts by scalars and hence that ${\rm proj}\,\rho$ factors through
$Gal(E/F)$. Using $\iota$ we get a projective representation to
${\rm PGL}_{n}(\mathbb{C})$ which we may lift as before to a finite
image representation of $G_{F}$; using $\iota$ again we get a finite
image irreducible representation $\tau\,:\, G_{F}\rightarrow{\rm GL}_{n}(\overline{\mathbb{Q}}_{\ell})$
with ${\rm proj\,\rho={\rm proj\,\tau}}$, and we may conclude that
there is a character $\epsilon\,:\, G_{F}\rightarrow\overline{\mathbb{Q}}_{\ell}^{\times}$
such that $\rho=\tau\otimes\epsilon$.\end{proof}
\begin{rem}
Of course the use of $\iota$ in the proof of 2) is unnecessary (and,
as the author is well aware of, may seem offensive to some). A short
argument with projective representations of finite groups shows that
${\rm proj\,\rho}$ takes values in ${\rm PGL}_{n}(\overline{\mathbb{Q}})$
and so rather than using $\iota$ it is enough to use the embeddings
$\overline{\mathbb{Q}}\hookrightarrow\mathbb{C}$ and $\overline{\mathbb{Q}}\hookrightarrow\overline{\mathbb{Q}}_{\ell}$.
\end{rem}
The next theorem we record is also well known, see e.g. \cite[\S 1, \S 8]{Far} and the references within. The corollary is essentially
(a weak version of) \cite[Proposition 7.12]{Far}; we give a proof
based on Proposition \ref{prop: primitive}.
\begin{thm}
\label{thm: characters}There exists a bijective correspondence between
algebraic characters of $W_{F}$ and geometric $\ell$-adic characters
of $G_{F}$ characterized by the property that if $\chi\leftrightarrow\epsilon$,
then for $v\nmid\ell\infty$, $\chi|_{W_{F_{v}}}$ is unramified if
and only if $\epsilon|_{G_{F_{v}}}$ is unramified, and $\chi(Frob_{v})=\iota(\epsilon(Frob_{v}))\in\overline{\mathbb{Q}}$
. 

Moreover, if $\chi\leftrightarrow\epsilon$, there exists a (necessarily
unique) integer $w$, the weight of $\chi$ (or $\epsilon$), such
that:

1) For $v\nmid\ell\infty$ such that $\chi|_{W_{F_{v}}}$ is unramified,
all conjugates of $\chi(Frob_{v})$ have Archimedean absolute value
$q_{v}^{-w/2}$, where $q_{v}$ is the size of the residue field of
$F_{v}$.

2) For $v\mid\infty$, if $\chi|_{W_{\overline{F_{v}}}}$ has the
form $z\mapsto z^{p}\bar{z}^{q}$ (for some identification $W_{\overline{F_{v}}}\cong\mathbb{C}^{\times}$)
then $p+q=w$.\end{thm}
\begin{cor}\label{cor: correspondence}
1) There exists a bijective correspondence $r\mapsto\rho_{r}$, $\rho\mapsto r_{\rho}$
between irreducible algebraic Weil parameters and irreducible potentially
abelian geometric $\ell$-adic representations characterized by the
property that if $r\leftrightarrow\rho$, then for $v\nmid\ell\infty$
$r|_{W_{F_{v}}}$ is unramified if and only if $\rho|_{G_{F_{v}}}$
is unramified, and $tr(r(Frob_{v}))=\iota(tr(\rho(Frob_{v}))$.

2) There exists a bijective correspondence between semisimple algebraic
Weil parameters and semisimple potentially abelian geometric $\ell$-adic
representations characterized by the property that if $r\leftrightarrow\rho$,
then for $v\nmid\ell\infty$ $r|_{W_{F_{v}}}$ is unramified if and
only if $\rho|_{W_{F_{v}}}$ is unramified, and $tr(r(Frob_{v}))=\iota(tr(\rho(Frob_{v}))$.\end{cor}
\begin{proof}
2) follows immediately from 1). To prove 1), note that it is clear
for finite image representations on both sides; just use $\iota$.
For characters this is Theorem \ref{thm: characters}. Thus for primitive
representations we write $r=\sigma\otimes\chi$, $\rho=\tau\otimes\epsilon$
as in Proposition \ref{prop: primitive} and define $\rho_{r}=\rho_{\sigma}\otimes\rho_{\chi}$
and $r_{\rho}=r_{\tau}\otimes r_{\epsilon}$. Then for general irreducibles
$r={\rm Ind}\, r^{\prime}$, $\rho={\rm Ind}\,\rho^{\prime}$ with
$r^{\prime}$, $\rho^{\prime}$ primitive, set $\rho_{r}={\rm Ind}\,\rho_{r^{\prime}}$,
$r_{\rho}={\rm Ind}\, r_{\rho^{\prime}}$; it is clear that these
constructions do the job and are inverse to each other (note that
$r={\rm Ind}\, r^{\prime}$ is algebraic if and only if $r^{\prime}$
is algebraic).
\end{proof}

\subsection{Equidistribution laws}

We begin this section with a well-known fact on the $L$-functions
of Weil parameters, and hence of their associated Galois representations in the case when they are algebraic.

\begin{prop}
\label{prop: mero cont}Let $r\,:\, W_{F}\rightarrow{\rm GL}_{n}(\mathbb{C})$
be an irreducible unitary Weil parameter. Then the Artin $L$-function
$L(s,r)$ has meromorphic continuation to all of $\mathbb{C}$.
Moreover it is invertible (i.e. holomorphic and non-vanishing) on
$Re\, s\geq1$ unless $r=|\cdot|^{i\lambda}$ for some $\lambda\in\mathbb{R}$.\end{prop}
\begin{proof}
When $n=1$ this is well known by class field theory. Let $n\geq2$.
If $r$ is primitive then we write $r=\sigma\otimes\chi$ as in Proposition
\ref{prop: primitive}. Assume first that $\chi$ is $\epsilon|\cdot|^{i\lambda}$
for some finite order character $\epsilon$ and $\lambda\in\mathbb{R}$.
Then without loss of generality $\epsilon=1$ (just incorporate it
in $\sigma$) and $L(s,r)=L(s+i\lambda,\sigma)$ so it is enough to
treat the case $r=\sigma$ finite image. By Brauer's theorem we may
write $r$ as a virtual direct sum 
\[
r=\bigoplus_{i\in I}\left({\rm Ind}_{W_{E_{i}}}^{W_{F}}\epsilon_{i}\right)^{\oplus a_{i}}
\]
where the $E_{i}/F$ are finite, $\epsilon_{i}$ is a finite order
character, $a_{i}\in\mathbb{Z}$ and $I$ is a finite indexing set.
Thus $L(s,r)=\prod_{i}L(s,\epsilon_{i})^{a_{i}}$ from which we deduce
the meromorphic continuation and the invertibility, except possibly
at $s=1$. The order of vanishing at $s=1$ is $-\sum_{i\in T}a_{i}$
where the sum runs over the set $T\subseteq I$ of $i$ such that
$\epsilon_{i}=1$. By character theory of finite groups and Frobenius
reciprocity 
\[
0=\left\langle tr\, r,1\right\rangle =\sum_{i\in I}a_{i}\left\langle \epsilon_{i},1\right\rangle =\sum_{i\in T}a_{i}
\]
where $\left\langle -,-\right\rangle $ denotes the usual inner product
of characters, hence $L(s,r)$ is invertible at $s=1$ as well.

Next if $\chi$ is not of the form $\epsilon|\cdot|^{i\lambda}$,
then it is not of this form when restricted to any open finite index
subgroup. We have $r=\sigma\otimes\chi$ and use Brauer's theorem
to write 
\[
\sigma=\bigoplus_{i\in I}\left({\rm Ind}_{W_{E_{i}}}^{W_{F}}\epsilon_{i}\right)^{\oplus a_{i}}
\]
with notation as before, hence 
\[
r=\bigoplus_{i\in I}\left({\rm Ind}_{W_{E_{i}}}^{W_{F}}(\epsilon_{i}\otimes\chi|_{W_{E_{i}}})\right)^{\oplus a_{i}}
\]
and $L(s,r)=\prod_{i}L(s,\epsilon_{i}\otimes\chi|_{W_{E_{i}}})^{a_{i}}$
and none of the $\epsilon_{i}\otimes\chi|_{W_{E_{i}}}$ are of the
form $|\cdot|^{i\lambda}$ which allows us to deduce the proposition
in this case.

Finally when $r$ is not primitive, write $r={\rm Ind}\, r^{\prime}$
with $r^{\prime}$ primitive and then $L(s,r)=L(s,r^{\prime})$ (note
that inducing $|\cdot|^{i\lambda}$ never produces an irreducible
representation).
\end{proof}

Next, define $W_{F}^{1}=Ker\,|\cdot|$ and $W_{F}^{ab,1}=Ker(|\cdot|\,:\, W_{F}^{ab}\rightarrow\mathbb{R}_{>0})$.

\begin{lem}
\label{lem: equivalences}Let $r\,:\, W_{F}\rightarrow{\rm GL}_{n}(\mathbb{C})$
be a unitary Weil parameter. Let $E/F$ be a finite Galois extension
such that $r$ factors through $W_{E/F}$. Then the image of $r$
is compact, hence closed in ${\rm GL}_{n}(\mathbb{C})$, and the following
are equivalent:

1) $r(W_{E}^{ab,1})=r(W_{E}^{ab})$.

2) $r(W_{F}^{1})=r(W_{F})$.

3) Write $r|_{W_{E}}=\chi_{1}\oplus...\oplus\chi_{n}$. Then for any
$a_{1},...,a_{n}\in\mathbb{Z}$, $\chi_{1}^{a_{1}}...\chi_{n}^{a_{n}}\neq|\cdot|^{i\lambda}$
for all $\lambda\in\mathbb{R}^{\times}$ .\end{lem}
\begin{proof}
Recall that $W_{E}^{ab}\cong W_{E}^{ab,1}\times\mathbb{R}_{>0}$ and
that $W_{E}^{ab,1}$ is compact. Thus $r(W_{E}^{ab,1})$ is compact
and on $\mathbb{R}_{>0}$, $r$ is a unitary character and is hence
trivial or maps surjectively onto ${\rm U}(1)$. It follows that $r(W_{E}^{ab})$
is compact. Since $W_{E}^{ab}$ has finite index in $W_{E/F}$ it
follows that the image of $r$ is compact.

Let us now show that 1) and 2) are equivalent. To do this, we introduce
$W_{E/F}^{1}$ which is the kernel of the norm map on $W_{E/F}$.
It sits in an exact sequence 
\[
1\rightarrow W_{E}^{ab,1}\rightarrow W_{E/F}^{1}\rightarrow Gal(E/F)\rightarrow1.
\]
Note that $W_{E}^{ab,1}=W_{E/F}^{1}\cap W_{E}^{ab}$ is normal in
$W_{E/F}$, and also that 2) is equivalent to $r(W_{E/F}^{1})=r(W_{E/F})$.
We have $W_{E/F}/W_{E/F}^{1}\cong W_{E}^{ab}/W_{E}^{ab,1}\cong\mathbb{R}_{>0}$
and these surject onto $r(W_{E/F})/r(W_{E/F}^{1})$ and $r(W_{E}^{ab})/r(W_{E}^{ab,1})$,
so they are connected. Thus it suffices to show that if one is finite
then the other is as well. Assume 1). Then 
\[
Gal(E/F)\twoheadrightarrow\frac{r(W_{E/F})}{r(W_{E}^{ab})}=\frac{r(W_{E/F})}{r(W_{E}^{ab,1})}\twoheadrightarrow\frac{r(W_{E/F})}{r(W_{E/F}^{1})}
\]
so 2) holds. Conversely, assume 2). Then 
\[
Gal(E/F)\twoheadrightarrow\frac{r(W_{E/F}^{1})}{r(W_{E}^{ab,1})}=\frac{r(W_{E/F})}{r(W_{E}^{ab,1})}\supseteq\frac{r(W_{E}^{ab})}{r(W_{E}^{ab,1})}
\]
and hence 1) holds. 

To finish the proof of the proposition it suffices to show that 1)
and 3) are equivalent. First assume 1). $r|_{W_{E}}$ takes values
in ${\rm U}(1)^{n}$ and the integers $a_{1},...,a_{n}$ define character
$\psi$ such that $\psi\circ r=\chi_{1}^{a_{1}}...\chi_{n}^{a_{n}}$.
Then $\psi(r(W_{E}^{ab}))=\psi(r(W_{E}^{ab,1}))$ which implies $\psi\circ r\neq|\cdot|^{i\lambda}$
for all $\lambda\in\mathbb{R}^{\times}$ as $|\cdot|^{i\lambda}(W_{E}^{ab})\neq1=|\cdot|^{i\lambda}(W_{E}^{ab,1})$
for $\lambda\in\mathbb{R}^{\times}$. Thus 1) implies 3). Conversely,
assume that 3) holds and let $G=r(W_{E}^{ab})$. As $W_{E}^{ab,1}$
is compact, $r(W_{E}^{ab,1})$ is a closed subgroup of $G$. Assume
for contradiction that $G/r(W_{E}^{ab,1})\neq1$. As $\mathbb{R}_{>0}\twoheadrightarrow G/r(W_{E}^{ab,1})$
we must have $G/r(W_{E}^{ab,1})={\rm U}(1)$ and hence we may find
a non-trivial character $\psi$ on $G$ that is trivial on $r(W_{E}^{ab,1})$.
Thus $\psi\circ r$ is non-trivial but trivial on $W_{E}^{ab,1}$
and hence equal to $|\cdot|^{i\lambda}$ for some $\lambda\in\mathbb{R}^{\times}$.
As $G$ is a closed subgroup of ${\rm U}(1)^{n}$ we may extend $\psi$
to a character on ${\rm U}(1)^{n}$ and hence there are integers $a_{1},...,a_{n}$
such that $\chi_{1}^{a_{1}}...\chi_{n}^{a_{n}}=|\cdot|^{i\lambda}$,
a contradiction. This concludes the proof. 
\end{proof}

We may now prove an equidistribution result for Weil parameters.

\begin{thm}
\label{thm: equidist}Let $r\,:\, W_{F}\rightarrow{\rm GL}_{n}(\mathbb{C})$
be a unitary Weil parameter and put $G=r(W_{F})$ (which is compact
by Lemma \ref{lem: equivalences}). Assume that $G=r(W_{F}^{1})$.
Let $S$ be a finite set of places containing all infinite places
such that $r$ is unramified outside $S$. Then the sequence $(r(Frob_{v}))_{v\not\in S}$
is equidistributed in the set of conjugacy classes of $G$ with respect
to measure induced by the Haar probability measure on $G$.\end{thm}
\begin{proof}
As indicated, it suffices to show that for every irreducible non-trivial
representation $\rho\,:\, G\rightarrow{\rm GL}_{m}(\mathbb{C})$,
the Artin $L$-function $L(s,\rho\circ r)$ of the (necessarily irreducible)
unitary Weil parameter $\rho\circ r$ has meromorphic continuation
and is invertible on $Re\, s\geq1$. By Proposition \ref{prop: mero cont}
it suffices to show that $\rho\circ r\neq|\cdot|^{i\lambda}$ for
any $\lambda\in\mathbb{R}$. As $r$ surjects onto $G$ $\rho\circ r$
cannot be trivial, and since $(\rho\circ r)(W_{F}^{1})=(\rho\circ r)(W_{F})$,
$\rho\circ r\neq|\cdot|^{i\lambda}$ for all $\lambda\in\mathbb{R}^{\times}$.
\end{proof}

\subsection{Applications to Galois representations and abelian varieties}

In this section we apply Theorem \ref{thm: equidist} to get equidistribution
laws for abelian varieties that become isogenous over some field extension
to a product of abelian varieties with CM. First we give results purely
in terms of Galois representations.
\begin{lem}\label{lem: image}
Let $\rho\,:\, G_{F}\rightarrow{\rm GL}_{n}(\overline{\mathbb{Q}}_{\ell})$
be a semisimple potentially abelian geometric $\ell$-adic representation
that is pure of weight $w\in\mathbb{Z}$, and let $r$ be the associated
Weil parameter. Then

1) $r\otimes|\cdot|^{-w/2}$ is unitary.

2) $(r\otimes|\cdot|^{-w/2})(W_{F})=(r\otimes|\cdot|^{-w/2})(W_{F}^{1})$.\end{lem}
\begin{proof}
1) First note that if $E/F$ is any finite extension, then $\rho$
is pure of weight $w$ if and only $\rho|_{G_{E}}$ is pure of weight
$w$, and $r\otimes|\cdot|^{-w/2}$ is unitary if and only if $r|_{W_{E}}\otimes|\cdot|^{-w/2}$
is unitary (as being unitary is equivalent to compactness of the image,
by Lemma \ref{lem: equivalences} and the unitary trick). Thus we are reduced to the case
when $\rho$ is a sum of characters, which directly reduces to the
case of a single character. Then as is well known there exists a unique
$\mu\in\mathbb{R}$ such that $r\otimes|\cdot|^{-\mu/2}$ is unitary.
For $v\nmid\ell\infty$ such that $r$ is unramified, we must then
have $|\iota(\rho(Frob_{v}))|=|r(Frob_{v})|=q_{v}^{-\mu/2}$ (where
the absolute value denotes the complex absolute value) and hence $\mu=w$
as desired.

2) By Lemma \ref{lem: equivalences} we are again reduced to the case
when $\rho=\epsilon_{1}\oplus...\oplus\epsilon_{n}$ is a sum of characters
and if we write $r=\chi_{1}\oplus...\oplus\chi_{n}$, we want to show
that for any $a_{1},...,a_{n}\in\mathbb{Z}$, $(\chi_{1}|\cdot|^{-w/2})^{a_{1}}...(\chi_{n}|\cdot|^{-w/2})^{a_{n}}\neq|\cdot|^{i\lambda}$
for all $\lambda\in\mathbb{R}^{\times}$, i.e. that $\chi_{1}^{a_{1}}...\chi_{n}^{a_{n}}\neq|\cdot|^{(\sum a_{i})(w/2)+i\lambda}$.
Since each $\chi_{i}$ is algebraic, we see that the left hand side
is algebraic, but the right hand side is only algebraic if $\lambda=0$. \end{proof}

\begin{cor}
Let $\rho\,:\, G_{F}\rightarrow{\rm GL}_{n}(\overline{\mathbb{Q}}_{\ell})$
be a semisimple potentially abelian geometric $\ell$-adic representation
that is pure of weight $w\in\mathbb{Z}$, and let $S$ be a finite
set of places containing all infinite places and places above $\ell$
such that $\rho$ is unramified outside $S$. Let $r$ be the Weil parameter associated with $\rho$. Then the elements $r(Frob_{v})q_{v}^{w/2}$
all lie in some compact group $G\subseteq{\rm GL}_{n}(\mathbb{C})$
and are equidistributed in the space of conjugacy classes of $G$
(with respect to normalized Haar measure on $G$).\end{cor}

\begin{proof} By the Lemma we may apply Theorem \ref{thm: equidist} to $r\otimes|\cdot|^{-w/2}$.\end{proof}

Applying this to the Tate module of an  abelian variety that becomes
isogenous to a product of abelian varieties with CM over a finite
extension of $F$, we get an equidistribution result. From the definition of $G$ and the $r(Frob_{v})q^{w/2}$ it is not clear that this data agrees with $ST_{A}$ and the conjugacy classes defined in \S \ref{sec: ST-group}. Next, we will prove the Sato-Tate conjecture in the form of \cite{FKRS} for products (up to isogeny) of CM abelian varieties.

\begin{prop}\label{thm: D and F}
\label{cor: D or F}Let $A/F$ be an abelian variety that becomes
isogenous to a product of abelian varieties with CM over a finite
extension of $F$. Then the Sato-Tate conjecture holds for $A$. In particular, the Sato-Tate conjecture holds for abelian surfaces of type $\textbf{D}$ or $\textbf{F}$.  \end{prop}

\begin{proof}
We follow the general strategy outline in \S \ref{sec: strategy}. Let $R$ be an irreducible representation of $ST_{A}$. By the remarks in the last paragraphs of section \S \ref{sec: generalities} we may extend $R$ to an algebraic representation of $G_{\ell}^{Zar}$, which we also denote by $R$, and $R\circ \rho_{A,\ell}$ is a pure, semisimple, potentially abelian geometric Galois representation. By Corollary \ref{cor: correspondence}, Theorem \ref{thm: equidist} and Lemma \ref{lem: image} we may now conclude.
\end{proof}

\begin{rem} Let $A/F$ be an abelian variety satisfying the hypotheses of Proposition \ref{thm: D and F}, with $\ell$-adic Tate module $\rho=\rho_{A,\ell}$ and corresponding Weil parameter $r$. Associated with $A$ we have groups $G$ and $ST_{A}$ together with conjugacy classes $(r(Frob_{v})q^{1/2})_{v}$  and $(s(v))_{v}$, and we have proved equidistribution results for both. As elements of ${\rm GL}_{2g}(\mathbb{C})$ $r(Frob_{v})q^{1/2}$ and $s(v)$ have the same characteristic polynomial. It is not clear (at least to the author) that $G=ST_{A}$ and  $r(Frob_{v})q^{1/2}=s(v)$, though it seems a natural guess. Unfortunately we have not been able to prove this equality. The issue (at least for the author of this paper) seems to be the inexplicit nature of the construction of the global Weil group for number fields.
\end{rem}

\section{\label{sec: GL(2)-type}Abelian Surfaces potentially of ${\rm GL}_{2}$-type
over totally real fields}

In this section we will deal with a number of cases where the abelian
surface $A/F$ has a two-dimensional commutative semisimple $\mathbb{Q}$-algebra  $K\subseteq{\rm End}_{\overline{\mathbb{Q}}}(A)^{0}$.
For technical reasons we need to assume that $K\subseteq{\rm End}_{F^{\prime}}(A)^{0}$
for a certain totally real field $F^{\prime}\supseteq F$ (forcing $F$ to be totally real
as well). This is analogous to the restriction to elliptic curves
over totally real fields in our current knowledge of the Sato-Tate
conjecture for non-CM elliptic curves. We will start off by recalling
some generalities from \cite{Rib} on the Tate modules of abelian
varieties of ${\rm GL}_{2}$-type, and then we will prove the Sato-Tate
conjecture case by case for cases $\textbf{B}$, $\textbf{C}$ and
$\textbf{E}$ under the above mentioned hypotheses.

\subsection{\label{sec: generalities}Some generalities}

To simplify terminology, let us make the following non-standard definition (cf.
\cite[\S 2]{Rib} in the case $F=\mathbb{Q}$ for the standard definition).
\begin{defn}\label{gl2def}
Let $A/F$ be an abelian variety. We say that $A$
is of ${\rm GL}_{2}$-type if it is isogenous over $F$ to a product $A_{1}\times ... \times A_{r}$ of simple abelian varieties $A_{i}$ over $F$ and there is a number field $K_{i}\subseteq {\rm End}_{F}(A_{i})^{0}$ of degree ${\rm dim}\, A_{i}$ for each $i=1,...,r$. 
\end{defn}

We will write $K$ for the product $\prod_{i=1}^{r}K_{i}$ viewed as a subalgebra of ${\rm End}_{F}(A)^{0}$. If $A/F$ is an abelian variety, note that if there is a commutative semisimple subalgebra $K\subseteq {\rm End}_{F}(A)^{0}$ of dimension ${\rm dim}\, A$ then $A/F$ is of ${\rm GL}_{2}$-type if each factor of $A$ corresponding to a simple factor of $K$ is simple over $F$. 

\begin{rem}
With the usual definition of ${\rm GL}_{2}$-type (corresponding to requiring $r=1$ in the above definition) one would say that an abelian variety $A$ satisfying the above definition is isogenous to a product of abelian varieties of ${\rm GL}_{2}$-type. We have chosen to relax the standard definition to avoid the cumbersome termonlogy "isogenous to a product of abelian varieties of ${\rm GL}_{2}$-type"; we apologize for this but hope that it will not cause the reader too much confusion.
\end{rem}

Let $A$ be an abelian surface of ${\rm GL}_{2}$-type over a totally
real field $F$ with $K\subseteq{\rm End}_{F}(A)^{0}$ of a two-dimensional commutative semisimple algebra. We have two cases corresponding to whether $A$ is simple and $K$ is a field, or $K=\mathbb{Q} \times \mathbb{Q}$ and $A$ is isogenous to a product of elliptic curves. If $A$ is simple, the following
results are recorded in \cite[\S 3]{Rib} when $F=\mathbb{Q}$,
but they extend with the same proofs to general totally real $F$.
We refer to \cite[\S 5]{BLGGT} for the definition of a
weakly compatible system.

\begin{prop}
\label{prop: compatible}Let $A$ be a simple abelian surface of ${\rm GL}_{2}$-type over a totally real field $F$ and let $K\subseteq{\rm End}_{F}(A)^{0}$
be a quadratic field. Let $\ell$ be a prime number. The $\ell$-adic
(rational) Tate module $\rho_{A,\ell}\,:\, G_{F}\rightarrow{\rm GL}_{4}(\overline{\mathbb{Q}}_{\ell})$
lands inside ${\rm GL}_{2}(K\otimes_{\mathbb{Q}}\overline{\mathbb{Q}}_{\ell})$
and we may decompose it into two two-dimensional pieces using the
two embeddings $K\hookrightarrow\overline{\mathbb{Q}}_{\ell}$. The
two-dimensional pieces fit together into a weakly compatible system
$(\rho_{A,\lambda}\,:\, G_{F}\rightarrow{\rm GL}_{2}(\overline{K}_{\lambda}))_{\lambda}$
where $\lambda$ ranges over the finite places of $K$ (or just embeddings
$K\hookrightarrow\overline{\mathbb{Q}}_{\ell}$ for all primes $\ell$).

In the following we use the association $\ell\leftrightarrow\lambda$
without further comment. The compatible system $(\rho_{A,\lambda})_{\lambda}$
satisfies the following properties (see \cite[Lemmas 3.1 and 3.2]{Rib} for properties 1) and 2)):

1) Assume that $K$ is a real quadratic field. Then $\det\,\rho_{A,\lambda}=\chi_{\ell}$,
the $\ell$-adic cyclotomic character.

2) Assume that $K$ is an imaginary quadratic field. Then $\det\,\rho_{A,\lambda}=\epsilon\chi_{\ell}$,
where $\epsilon\,:\, G_{F}\rightarrow K^{\times}$ is a finite order
character (independent of $\lambda$).

3) In either case $\rho_{A,\lambda}$ is totally odd, regular with
Hodge-Tate weights $0$ and $-1$ (independent of embeddings and choice
of places), and the compatible system is pure of weight $-1$.
\end{prop}

When $A$ is isogenous over $F$ to a product of elliptic curves the obvious analogue of parts (1) and (3) of the above Proposition holds (one might view the algebra $\mathbb{Q} \times \mathbb{Q}$ as "totally real"). We also record a conjugate self-duality result in the case when
$K$ is imaginary quadratic. 
\begin{prop}
\label{prop: conj selfdual}(\cite[Proposition 3.4]{Rib}) We use the
notation of Proposition \ref{prop: compatible} and assume in addition
that $K$ is imaginary quadratic. Then $\rho_{A,\lambda_{1}}\cong\rho_{A,\lambda_{2}}\otimes\epsilon$,
where $\epsilon$ denotes the character in Proposition \ref{prop: compatible}
2) and the $\lambda_{i}$ are the two distinct embeddings $K\hookrightarrow\overline{\mathbb{Q}}_{\ell}$.
\end{prop}

Finally, let us recall the following from the discussion leading up to the definition of the Sato-Tate group: We have an exact sequence 
\[
1\rightarrow G_{\ell}^{1,Zar} \rightarrow G_{\ell}^{Zar} \rightarrow \mathbb{G}_{m}\rightarrow 1
\]
and $G_{\ell}^{Zar}$ contains the central $\mathbb{G}_{m}$ inside ${\rm GL}_{4}$. Since the similitude character becomes $z\mapsto z^{2}$ when restricted to the central $\mathbb{G}_{m}$ we have an isogeny $G_{\ell}^{1,Zar}\times \mathbb{G}_{m} \rightarrow G_{\ell}^{Zar}$. This gives us a surjection $G^{1}\times \mathbb{C}^{\times} \rightarrow G$ on $\mathbb{C}$-points with kernel $\{1,(-1,-1)\}$. Since $ST_{A}$ is a compact real form of the complex reductive group $G^{1}$ we see that any (complex, finite-dimensional) irreducible representation of $ST_{A}$ may be extended to an irreducible algebraic representation of $G$, hence of $G_{\ell}^{Zar}(\overline{\mathbb{Q}}_{\ell})$, by choosing a compatible character of the central $\mathbb{G}_{m}$.

This has the following important consequence: Let $r$ be an irreducible representation of $ST_{A}$ and let $w\in \mathbb{Z}$ be such that $r(-1)=(-1)^{w}$. Then $z\mapsto z^{w}$ is compatible with $r$ and the data $(r,w)$ defines an irreducible algebraic representation $R(w)$ of $G_{\ell}^{Zar}(\overline{\mathbb{Q}}_{\ell})$. Thus we get a weakly compatible system of Galois representations $R(w)\circ \rho_{A,\ell}$ that are essentially self-dual, totally odd and irreducible (importantly) and satisfies
\[
L^{S}(s,r,A)=L^{S}(s+w/2,R(w)\circ \rho_{A,\ell})
\]
(directly from the definition of the two Euler products).
\subsection{Galois type $\textbf{B}[C_{1}]$}

Recall that we are using the presentation ${\rm USp}(4)=\left\{ A\in{\rm U}(4)\mid A^{t}S^{\prime}A=S^{\prime}\right\} $.
Case $\textbf{B}$ corresponds to $ST_{A}^{0}={\rm SU}(2)\times{\rm SU}(2)$,
which we may embed into ${\rm USp}(4)$ by 
\[
(A,B)\mapsto\left(\begin{array}{cc}
A & 0\\
0 & B
\end{array}\right)
\]
The reference for the group theory is \cite[\S 3.6]{FKRS}.

In the case of Galois type $\textbf{B}[C_{1}]$ the Sato-Tate group
is simply ${\rm SU}(2)\times{\rm SU}(2)$ and its irreducible representations
are $r_{k,l}=Sym^{k}(St)\otimes Sym^{l}(St)$ for nonnegative integers
$k$, $l$, where $St$ denotes the standard two-dimensional representation
of any subgroup of ${\rm GL}_{2}$. Let $A/F$ be an abelian surface
of Galois type $\textbf{B}[C_{1}]$. There are two cases to consider;
either $A$ is isogenous to the product of elliptic curves over $F$
without CM that do not become isogenous over any finite extension
of $F$, or $A$ has real multiplication defined over $F$ by a real
quadratic field $K$. In the first case, the Sato-Tate conjecture
is known and due to Harris (\cite[Theorem 5.4]{Har}), modulo certain
hypotheses that have subsequently been verified (see \cite{BGHT}
for a discussion and references). It remains to deal with the second
case. The proof is the same as that of Harris in the first case (up
to a few minor details and the general remark that we use more powerful
potential automorphy theorems).
\begin{prop}
\label{prop: B[C_1]}Let $A$ be an abelian surface over $F$ such
that ${\rm End}_{F}(A)^{0}={\rm End}_{\overline{\mathbb{Q}}}(A)^{0}=K$
is a two-dimensional semisimple algebra over $\mathbb{Q}$. Then the
Sato-Tate conjecture holds for $A$.\end{prop}
\begin{proof}
As mentioned above, $K$ is either $\mathbb{Q}\times\mathbb{Q}$ or
a real quadratic field, the first case being treated by Harris, so
we may assume that $K$ is real quadratic (the proof below also works with the obvious modifications). Consider the regular, totally odd and essentially self-dual weakly compatible system $(\rho_{A,\lambda})_{\lambda}$ given by Proposition \ref{prop: compatible}. For any $m\geq 0$, the weakly compatible system $(Sym^{m}\rho_{A,\lambda})_{\lambda}$ inherits these properties and moreover is irreducible (since the $\rho_{A,\lambda}(G_{F})$ is open in its Zariski closure, which is ${\rm GL}_{2}(\overline{\mb{Q}}_{\ell})$) so we may apply \cite[Theorem 5.4.1]{BLGGT} to deduce that there is a Galois, totally real extension $F^{\prime}/F$ such that $(Sym^{m}\rho_{A,\lambda})_{\lambda}$ becomes cuspidal automorphic after restriction to $F^{\prime}$. 

Next, let us consider the partial $L$-function  
\[ L^{S}(s,r_{k,l},A)=\prod_{v\notin S}\frac{1}{{\rm det}(1-r_{k,l}(s(v))q_{v}^{-s})} \] 
which we want to show is invertible on $Re\, s\geq1$ when $(k,l)\neq(0,0)$. We may extend $r_{k,l}$ to a representation $R_{k,l}$ of $G_{\ell}^{Zar}$ as described in the previous section. Here let us be explicit and, since $G_{\ell}^{Zar}\subseteq {\rm GL}_{2}(\overline{\mathbb{Q}}_{\ell})\times{\rm GL}_{2}(\overline{\mathbb{Q}}_{\ell})$, define $R_{k,l}$ by restricting $Sym^{k}\otimes Sym^{l}$. Then we have
\[ L^{S}(s,R_{k,l},A)=L^{S}(s+(k+l)/2,R_{k,l}\circ r_{A,\ell})=L^{S}(s+(k+l)/2,Sym^{k}\rho_{A,\lambda_{1}}\otimes Sym^{l}\rho_{A,\lambda_{2}}) \] 
where we let $\lambda_{1}$ and $\lambda_{2}$ denote the two embeddings $K\hookrightarrow\overline{\mathbb{Q}}_{\ell}$. By Brauer's theorem, as in e.g. the proof of \cite[Theorem 4.2]{HSBT}, the invertibility of this $L$-function follows from that of \[ L^{S}(s+(k+l)/2,Sym^{k}(\rho_{A,\lambda_{1}}|_{G_{E}})\otimes Sym^{l}(\rho_{A,\lambda_{2}}|_{G_{E}})) \] for arbitrary subextensions $F^{\prime}/E/F$ such that $F^{\prime}/E$ is solvable, and cuspidal automorphy of $(Sym^{m}\rho_{A,\lambda})_{\lambda}$ over $F^{\prime}$ implies cuspidal automorphy over $E$ by the theory of cyclic base change (\cite{AC}). Thus, Rankin-Selberg theory (as in the proof of \cite[Theorem 5.3]{Har}) implies that $L^{S}(s+(k+l)/2,Sym^{k}(\rho_{A,\lambda_{1}}|_{G_{E}})\otimes Sym^{l}(\rho_{A,\lambda_{2}}|_{G_{E}}))$ is invertible, as long as $Sym^{k}(\rho_{A,\lambda_{1}}|_{G_{E}})\not\cong Sym^{l}(\rho_{A,\lambda_{2}}|_{G_{E}})$ (use essential self-duality and the fact that the multiplier of $\rho_{A,\lambda}$ is the cyclotomic character $\chi_{\ell}$). Then, arguing as in the paragraph before \cite[Theorem 5.3]{Har} (using our knowledge of ${\rm End}_{\overline{\mathbb{Q}}}(A)$ and the open image theorem for the $\rho_{A,\lambda}$), we deduce that if $Sym^{k}(\rho_{A,\lambda_{1}}|_{G_{E}})\cong Sym^{l}(\rho_{A,\lambda_{2}}|_{G_{E}})$ then there exists a finite extension $E^{\prime}/F$ such that $\rho_{A,\lambda_{1}}|_{G_{E^{\prime}}}\cong\rho_{A,\lambda_{2}}|_{G_{E^{\prime}}}$, but this contradicts the fact that ${\rm End}_{\overline{\mathbb{Q}}}(A)^{0}=K$ by Faltings's theorem.
\end{proof}

\subsection{Galois type $\textbf{B}[C_{2}]$}

Here $ST_{A}=N({\rm SU}(2)\times{\rm SU}(2))=\left\langle {\rm SU}(2)\times{\rm SU}(2),J\right\rangle \subseteq{\rm USp}(4)$.
Note that $ST_{A}$ has two components as $J^{2}=Id_{4}$. Let us
start by recording the following standard consequence of Clifford's
Theorem which will be used numerous times throughout this paper:
\begin{lem}
\label{lem: Bump}Let $G$ and $H$ be topological groups with $H\subseteq G$
a subgroup of index $2$ and let $x\in G\backslash H$ be any element.
If $r$ is a continuous finite-dimensional irreducible representation
of $H$ we define its twist $r^{x}$ by $r^{x}(h)=r(xhx^{-1})$. Then
$r$ extends to a representation of $G$ if and only if $r\cong r^{x}$,
in which case it extends to exactly two non-isomorphic representations
(one is obtained from the other by twisting by the nontrivial character
of $G/H$) and all the other irreducible representations of $G$ are
of the form $Ind_{H}^{G}r$ where $r$ is an irreducible representation
of $H$ such that $r\not\cong r^{x}$.
\end{lem}
We will use it here with $H=ST_{A}^{0}$, $G=ST_{A}$. Given $(A,B)\in ST_{A}^{0}$,
we have that $J(A,B)J^{-1}=(J_{2}BJ_{2}^{-1},J_{2}AJ_{2}^{-1})$.
Let $r_{k,l}=Sym^{k}(St)\otimes Sym^{l}(St)$ be an irreducible representation
of $ST_{A}^{0}$. Then since $J_{2}\in{\rm SU}(2)$ we see that $r_{k,l}^{J}\cong r_{l,k}$.
Hence $r_{k,l}$ extends if and only if $k=l$ and the Lemma gives
us all the irreducible representations of $ST_{A}$. For $k\neq l$,
we put $s_{k,l}=Ind_{H}^{G}r_{k,l}$, for $k=l$ we have two representations
$s_{k}^{1}$, $s_{k}^{2}$ that extend $r_{k,k}$. We make the convention
that $s_{0}^{1}$ is the trivial representation.

Now let us consider an abelian surface $A/F$ of type $\textbf{B}[C_{2}]$.
This means that $L/F$ has degree two, ${\rm End}_{F}(A)^{0}=\mathbb{Q}$
and ${\rm End}_{L}(A)^{0}={\rm End}_{\overline{\mathbb{Q}}}(A)^{0}$
is a two-dimensional semisimple algebra over $\mathbb{Q}$. Thus $\rho_{A,\ell}$
is irreducible but $\rho_{A,\ell}|_{G_{L}}$ is a sum of two two-dimensional
irreducible representations that do not become either isomorphic or
reducible upon further extension of $F$.
\begin{prop}
\label{prop: B[C_2]}Let $A/F$ be an abelian surface of Galois type
$\textbf{B}[C_{2}]$ and assume that $L$ is totally real. Then the
Sato-Tate conjecture holds for $A$.\end{prop}
\begin{proof}
We have an exact sequence 
\[
1\rightarrow G_{\ell}^{Zar}(A_{L})(\overline{\mathbb{Q}}_{\ell}) \rightarrow G_{\ell}^{Zar}(A)(\overline{\mathbb{Q}}_{\ell}) \rightarrow \mathbb{Z}/2\mathbb{Z} \rightarrow 1.
\]
Consider first the representation $s_{k,l}$ for $k\neq l$. Note that $s_{k,l}(-1)=r_{k,l}(-1)=(-1)^{k+l}$, this gives us an algebraic representation of $G_{\ell}^{Zar}(A)(\overline{\mathbb{Q}}_{\ell})$, which we will denote by $S_{k,l}$, satisfying $S_{k,l}=Ind_{G_{\ell}^{Zar}(A_{L})(\overline{\mathbb{Q}}_{\ell})}^{G_{\ell}^{Zar}(A)(\overline{\mathbb{Q}}_{\ell})}R_{k,l}$  where the $R_{k,l}$ on the right hand side is the representation as in Proposition \ref{prop: B[C_1]}. Thus
\[
L^{S}(s,s_{k,l},A)=L^{S}(s+(k+l)/2,S_{k,l}\circ \rho_{A,\ell})=L^{S}(s,r_{k,l},A_{L})
\]
and hence it is invertible by the proof of Proposition \ref{prop: B[C_1]}. 

Next we let $k\geq 1$ and consider $s_{k}^{i}$. We have $s_{k}^{i}(-1)=s_{k,k}(-1)=(-1)^{2k}$ so we may choose $w=2k$ and extend $R_{k}^{i}$ to an algebraic representation $S_{k}^{i}$ of $G_{\ell}^{Zar}(A)(\overline{\mathbb{Q}}_{\ell})$ with $L^{S}(s,s_{k}^{i},A)=L^{S}(s+k,S_{k}^{i}\circ \rho_{A,\ell})$. The invertibility of this $L$-function follows from that of
$L^{S}(s+k,S_{k}^{i}\circ\rho_{A,\ell}|_{G_{E}})=L^{S}(s+k,r_{k,k}\circ\rho_{A,\ell}|_{G_{E}})$ by Brauer's theorem,
where $E$ is a subextension of a totally real Galois extension $F^{\prime}/L$, and moreover $F^{\prime}/F$ is Galois as well. We now apply the proof of Proposition \ref{prop: B[C_1]} to conclude, noting that the extra condition that $F^{\prime}/F$ is Galois as well is allowed in \cite[Theorem 5.4.1]{BLGGT}.

Finally we need to consider $L^{S}(s,s_{0}^{2},A)$. Since the composition
$G_{F}\rightarrow\ G_{\ell}^{Zar}(A)(\overline{\mathbb{Q}}_{\ell}) \rightarrow\mathbb{Z}/2\mathbb{Z}$
is surjective with kernel $G_{L}$ $L^{S}(s,s_{0}^{2},A)$ is the
Hecke $L$-function associated with the nontrivial character of $Gal(L/F)$, hence invertible.
\end{proof}

\subsection{Type $\textbf{C}$ }

In this case only the Galois type $\textbf{C}[C_{2}]$ is possible
over a totally real field. We are still using the presentation ${\rm USp}(4)=\left\{ A\in{\rm U}(4)\mid A^{t}S^{\prime}A=S^{\prime}\right\} $.
We let ${\rm U}(1)\times{\rm SU}(2)$ be embedded into ${\rm USp}(4)$
by 
\[
(u,A)\mapsto\left(\begin{array}{ccc}
u\\
 & \overline{u}\\
 &  & A
\end{array}\right).
\]
$ST_{A}$ may then be described as $\left\langle {\rm U}(1)\times{\rm SU}(2),a\right\rangle $.
Thus $ST_{A}=N({\rm U}(1))\times{\rm SU}(2)$ where $N({\rm U}(1))$
denotes the normalizer of ${\rm U}(1)$ in ${\rm SU}(2)$, embedded
via 
\[
u\mapsto\left(\begin{array}{cc}
u\\
 & \overline{u}
\end{array}\right).
\]
The irreducible non-trivial representations of $N({\rm U}(1))$ are
of the form $r_{k}=Ind_{{\rm U}(1)}^{N({\rm U}(1))}(u\mapsto u^{k})$
for $k\in\mathbb{Z}$ nonzero (use e.g. Lemma \ref{lem: Bump}). If
we let $r_{0}^{0}$ denote the trivial representation of $N({\rm U}(1))$
and $r_{0}^{1}$ denotes the lift of the nontrivial character of $N({\rm U}(1))/{\rm U}(1)$,
we deduce that the irreducible representations of $ST_{A}$ are of
the form $r_{k}\otimes Sym^{l}St$ for $k\in\mathbb{Z}\backslash\left\{ 0\right\} $,
$l\in\mathbb{Z}_{\geq0}$ and $r_{0}^{\epsilon}\otimes Sym^{l}St$
for $\epsilon\in\{0,1\}$.

Consider an abelian variety $A/F$ of type $\textbf{C}[C_{2}]$. $A$
is then isogenous to a product $E_{1}\times E_{2}$ of elliptic curves
over $F$ (\cite[Proposition 4.5]{FKRS}) where we may take $E_{1}$
to have CM (defined over $L$ but not over $F$) and $E_{2}$ to have no CM. 
\begin{prop}
Let $A/F$ be an abelian surface of Galois type $\textbf{C}[C_{2}]$
and assume that $F$ is totally real (and hence $L$ is a CM field).
Then the Sato-Tate conjecture holds for $A$.\end{prop}
\begin{proof}
We have $\rho_{A,\ell}=\rho_{E_{1},\ell}\oplus\rho_{E_{2},\ell}$
and $\rho_{E_{1},\ell}=Ind_{G_{L}}^{G_{F}}\psi_{\ell}$ for some weakly
compatible system of weight one essentially conjugate-self-dual algebraic
Hecke characters $\psi_{\ell}$. From this one sees that for $k\geq1$,
$l\geq1$, 
\[
L^{S}(s,r_{k}\otimes Sym^{l}St,A)=L^{S}(s+(k+l)/2,Sym^{l}(\rho_{E_{2},\ell}|_{G_{L}})\otimes\psi_{\ell}^{k})
\]
and the $(Sym^{l}(\rho_{E_{2},\ell}|_{G_{L}})\otimes\psi_{\ell}^{k})_{\ell}$
form a weakly compatible system of irreducible essentially
conjugate-self-dual regular Galois representations which we may apply \cite[Theorem 5.4.1]{BLGGT} to to deduce that it is
potentially automorphic; invertibility of $L^{S}(s,r_{k}\otimes Sym^{l}St,A)$
follows as usual from Brauer's theorem. For $k\geq1$, $l=0$ we have
\[
L^{S}(s,r_{k}\otimes1,A)=L^{S}(s+l/2,\psi_{\ell}^{k})
\]
and the result follows. For $l\geq1$ , $L^{S}(s,r_{0}^{\epsilon}\otimes Sym^{l}St,A)=L^{S}(s,Sym^{l}r_{E_{2},\ell}\otimes\chi)$
where $\chi$ is a character of $Gal(L/F)$ (trivial if and only if
$\epsilon=0$); invertibility follows from potential automorphy of
$Sym^{l}r_{E_{2},\ell}$ and Brauer's theorem. Finally, $L^{S}(s,r_{0}^{1}\otimes1,A)$
is the $L$-function of the nontrivial character of $Gal(L/F)$, hence
invertible. 
\end{proof}

\subsection{Galois types $\textbf{E}[C_{n}]$, $n=1,3,4,6$, $\textbf{E}[C_{2},\mathbb{R}\times \mathbb{R}]$ and $\textbf{E}[C_{2},\mathbb{C}]$}

We now switch presentation of ${\rm USp}(4)$ and instead use ${\rm USp}(4)=\left\{ A\in{\rm U}(4)\mid A^{t}SA=S\right\} $.
For type $E$ one has $ST_{A}^{0}={\rm SU}(2)$. We embed ${\rm U}(2)$
(and hence ${\rm SU}(2)$) into ${\rm USp}(4)$ via 
\[
A\mapsto\left(\begin{array}{cc}
A\\
 & \overline{A}
\end{array}\right).
\]

The Galois types $\textbf{E}[C_{n}]$, $n=1,3,4,6$, $\textbf{E}[C_{2},\mathbb{R}\times \mathbb{R}]$ and $\textbf{E}[C_{2},\mathbb{C}]$
correspond to the Sato-Tate groups $E_{n}$, $n=1,3,4,6$, $J(E_{1})$ and $E_{2}$
respectively, where the groups $E_{n}$ are defined as $\left\langle {\rm SU}(2),e^{\pi i/n}\right\rangle $ and $J(E_{1})=\langle {\rm SU}(2),J \rangle$. Here $e^{\pi i/n}$ denotes the element $diag(e^{\pi i/n},e^{\pi i/n},e^{-\pi i/n},e^{-\pi i/n})\in{\rm USp}(4)$. As $J$ commutes with ${\rm SU}(2)$ $J(E_{1})$ is isomorphic to ${\rm SU}(2)\times C_{2}$.

For the groups $E_{n}$ the irreducible representations are of the form $r_{k,w}=Sym^{k}St\otimes\chi_{w}$,
where $k\in\mathbb{Z}_{\geq0}$, $w\in\mathbb{Z}/2n\mathbb{Z}$, $k\equiv w\,(2)$
and $\chi_{w}$ is the character $z\mapsto z^{w}$ of $\mu_{2n}$.
$E_{n}$ is then the group 
\[
{\rm U}_{n}(2)=\left\{ A\in{\rm U}(2)\mid\det\, A\in\mu_{n}\right\} 
\]
embedded in ${\rm USp}(4)$ via the above embedding of ${\rm U}(2)$. For $J(E_{1})$ the irreducible representations are of the form $r_{k,w}=Sym^{k}St \otimes \chi_{w}$, $k$ integer $\geq 0$ and $w\in {0,1}$, with $w=0$ corresponding to the trivial character of $C_{2}$.

Arithmetically these cases correspond to abelian surfaces $A$ such that
${\rm End}_{L}(A)^{0}$ is a quaternion algebra over $\mathbb{Q}$
and, when $F\neq L$, ${\rm End}_{F}(A)^{0}$ is a quadratic
field or $\mathbb{Q}\times \mathbb{Q}$. This follows from \cite[Proposition 4.7]{FKRS},
after we have made some remarks regarding the formulation of that
proposition, assuming (in their notation) $C=\mathbb{Q}$ which is
all we need. First, note that the assumption that $E$ becomes isogenous
over $L$ to a product of elliptic curves is not necessary and indeed
not used in the proof, all that is used is that ${\rm End}_{L}(A)^{0}$
is a quaternion algebra. Moreover, there is a possible source of confusion in the formulation in that one has to allow the ``quadratic extension''
mentioned in cases (i), $n=2$ and (ii) to be $\mathbb{Q}\times\mathbb{Q}$
(e.g. $L/F$ quadratic, $A=E_{1}\times E_{2}$ where $E_{1}$ and
$E_{2}$ do not have CM and are non-isomorphic quadratic twists of
each other with respect to $L/F$).

Thus we see that $\rho_{A,\ell}$ is a direct sum of two two-dimensional
representations $\rho_{A,\lambda_{1}}$ and $\rho_{A,\lambda_{2}}$
and that there is a finite order character $\epsilon$ such that $\rho_{A,\lambda_{1}}\cong\rho_{A,\lambda_{2}}\otimes\epsilon$.
Since ${\rm End}_{L}(A)^{0}$ is a quaternion algebra, $\rho_{A,\lambda_{1}}$
and $\rho_{A,\lambda_{2}}$ become isomorphic when restricted to $L$
and are irreducible after restriction to any finite extension of $F$.
Thus $\rho_{A,\lambda_{1}}=\rho_{A,\lambda_{2}}\otimes\epsilon$ for
some character $\epsilon$ of $Gal(L/F)$. Note that we saw this in
Proposition \ref{prop: compatible} when ${\rm End}_{F}(A)^{0}$ is an imaginary quadratic field, and that $\det\,\rho_{A,\lambda_{1}}=\epsilon\chi_{\ell}$.

\begin{prop}
\label{prop: E[C_n]}Let $A$ be an abelian surface of Galois type
$\textbf{E}[C_{n}]$ for $n=1,3,4,6$ or  $\textbf{E}[C_{2},\mathbb{C}]$
and assume that $F$ is totally real. Then $A$ satisfies the Sato-Tate
conjecture.\end{prop}
\begin{proof}

Pick a finite place $\lambda$ of $K={\rm End}_{F}(A)^{0}$ (which is always imaginary quadratic by assumption, cf. \cite[Proposition 4.7]{FKRS}) and consider the two-dimensional representation $\rho_{A,\lambda}$ which is a direct summand of $\rho_{A,\ell}$. We claim that $G^{1,Zar}_{\ell}(A)$ is isomorphic to the Zariski closure $G_{\lambda}^{1,Zar}$ of $\rho_{A,\lambda}(G_{F}^{1})$. This follows since $\rho_{A,\ell}=\rho_{A,\lambda}\oplus \rho_{A,\bar{\lambda}}$ and after restriction to $G_{F}^{1}$, $\rho_{A,\bar{\lambda}}\cong \rho_{A,\lambda}\otimes det(\rho_{A,\lambda})^{-1}$.

Consider the representation $r_{k,w}$ of $ST_{A}$. $-1\in ST_{A}$ acts by $(-1)^{k}$ and hence we may define an irreducible algebraic representation $R_{k,w}$ of $G_{\ell}^{Zar}$ such that 
\[
L^{S}(s,r_{k,w},A)=L^{S}(s+k/2,R_{k,w}\circ \rho_{A,\lambda}).
\]
The weakly compatible system $(R_{k,w}\circ \rho_{A,\lambda})_{\lambda}$ is irreducible, totally odd, essentially self-dual and also regular (the latter is easily seen by restricting to $Gal(\overline{F}/L)$). Hence we deduce invertibility of $L^{S}(s+k/2,R_{k,w}\circ \rho_{A,\ell})$ in the case $k\geq 1$ in the usual way from \cite[Theorem 5.4.1]{BLGGT} and Brauer's Theorem. When $k=0$ and $w\neq 0$ invertibility follows from the fact that $R_{k,w}\circ \rho_{A,\ell}$ is a non-trivial character of $Gal(L/F)$.
\end{proof}

\begin{prop}
\label{prop: E[R]}Let $A$ be an abelian surface of Galois type
$\textbf{E}[C_{2},\mathbb{R}\times \mathbb{R}]$
and assume that $F$ is totally real. Then $A$ satisfies the Sato-Tate
conjecture.\end{prop}

\begin{proof}

Recall that $ST_{A}=J(E_{1})\cong {\rm SU}(2)\times C_{2}$ and the irreducible representations are of the form $r_{k,w}=Sym^{k}St \otimes \chi_{w}$, $k$ integer $\geq 0$ and $w\in {0,1}$, with $w=0$ corresponding to the trivial character of $C_{2}$. Moreover $\rho_{A,\ell}=\rho_{A,\lambda}\oplus \rho_{A,\bar{\lambda}}$ where $\lambda$ is a homomorphism ${\rm End}_{F}(A)^{0}\rightarrow \overline{\mathbb{Q}}_{\ell}$ and $\bar{\lambda}$ is the other one. We have $\rho_{A,\lambda}\cong \rho_{A,\bar{\lambda}}\otimes \epsilon$ where $\epsilon$ is the non-trivial character $Gal(L/F)\rightarrow \{\pm 1\}$. This presentation corresponds to conjugating $J(E_{1})$ inside ${\rm GL}_{4}(\mb{C})$ by e.g. the matrix 
\[
\left(\begin{array}{cc}
I & J_{2}\\
I & -J_{2}
\end{array}\right),
\]
this sends ${\rm SU}(2)$, embedded as 
\[
A\mapsto
\left(\begin{array}{cc}
A\\
& \bar{A}
\end{array}\right),
\]
to ${\rm SU}{2}$ embedded as
\[
A\mapsto
\left(\begin{array}{cc}
A\\
& A
\end{array}\right)
\]
and $J$ gets sent to $
\left(\begin{array}{cc}
I\\
& -I
\end{array}\right).
$ From this we see that the conjugacy classes $s_{v}$ are independent of $\ell$ as they are determined by the trace of $Frob_{v}$ in the weakly compatible system $\rho_{a,\lambda}$ in the case that ${\rm End}_{F}(A)^{0}$ is a field, and by the traces of $Frob_{v}$ in the two weakly compatible systems  $\rho_{a,\lambda_{1}}$ and  $\rho_{a,\lambda_{2}}$ in the case where ${\rm End}_{F}(A)^{0}\cong \mb{Q} \times \mb{Q}$, where $\lambda_{i}$ is the homomorphism $\mb{Q} \times \mb{Q} \rmap \overline{\mb{Q}}_{\ell}$ corresponding to projection to the $i$th factor. Now pick an extension $R_{k,w}$ of $r_{k,w}$ as usual, then the $R_{k,w}\circ \rho_{A,\ell}$ form an essentially self-dual, odd, regular and irreducible weakly compatible system (where compatibility follows from independence of the $s_{v}$) to which we may apply \cite[Theorem 5.4.1]{BLGGT} when $k\geq 1$ and deduce invertibility as usual, and for $k=0$, $w=1$ we proceed as usual as well. 
\end{proof}

\begin{rem} 
Proposition \ref{prop: E[R]} could also have been proven in similar fashion to Proposition \ref{prop: E[C_n]} (and vice versa). The proof given rests on the $\ell$-independence of the $s_{v}$ which is unknown in any significant generality, unlike the $\ell$-independence of $ST_{A}$ (which is of course a prerequisite).
\end{rem}

\subsection{Galois types $\textbf{E}[D_{n}]$,
$n=2,3,4,6$.}

Here $\textbf{E}[D_{n}]$ corresponds to $J(E_{n})$, $n=2,3,4,6$,
where the groups $J(E_{n})$ are defined as $\left\langle E_{n},J\right\rangle $.
Note that $J$ commutes with ${\rm SU}(2)$ and that $Je^{\pi i/n}J=e^{-\pi i/n}$.
Using the notation of the previous subsection we see that $r_{k,w}^{J}\cong r_{k,-w}$.
Thus $r_{k,w}$ extends if and only if $w=0$ or $n$; we get irreducible
representations $R_{k,0,\epsilon}$, $R_{k,n,\epsilon}$ ($k\in\mathbb{Z}_{\geq0}$,
$\epsilon\in\{0,1\}$, $k\equiv0$ resp. $n$ modulo $2$) and $R_{k,w}=Ind_{E_{n}}^{J(E_{n})}r_{k,w}$
($k\in\mathbb{Z}_{\geq0}$, $w\in(\mathbb{Z}/2n\mathbb{Z})\backslash\{0,n\}$,
$k\equiv w\,(2)$). In the case $n=2$ we may also describe the representations of $J(E_{2})$ using the subgroup $J(E_{1})$ instead of $E_{2}$; they divide into induced and extended as before. 

Let $A/F$ be an abelian surface of Galois type $\textbf{E}[D_{n}]$ for $n\in\{2,3,4,6\}$. By \cite[Proposition
4.7]{FKRS} there is a quadratic extension $F^{\prime}/F$
such that $A/F^{\prime}$ has type $\textbf{E}[C_{n}]$ if $n\in\{3,4,6\}$
and $\textbf{E}[C_{2},\mathbb{C}]$ or $\textbf{E}[C_{2},\mathbb{R}\times\mathbb{R}]$
if $n=2$ (in the latter case $F^{\prime}$ is not unique, but any
choice will do).
\begin{prop}
\label{prop: E[D_n]}Let $A/F$ be an abelian surface of Galois type
$\textbf{E}[D_{n}]$ for $n\in\{2,3,4,6\}$ and assume that $F^{\prime}$
is a totally real field. Then the Sato-Tate conjecture holds.\end{prop}
\begin{proof}
The proof follows the same lines as the proof of Proposition \ref{prop: B[C_2]}. If R is an irreducible representation of $ST_{A}$ of dimension $\geq 2$, then if $R$ is induced from $ST_{A_{F^{\prime}}}$ one reduces directly to the proof of Proposition \ref{prop: E[C_n]} for the invertibility of the corresponding $L$-function, and if $R$ is extended from $ST_{A_{F^{\prime}}}$ one argues as in the proof of Proposition \ref{prop: B[C_2]}, using the proof of Proposition \ref{prop: E[C_n]} or Proposition \ref{prop: E[R]} in place of that of Proposition \ref{prop: B[C_1]}. Finally if $\dim\,R=1$ we use surjectivity onto the component group as usual.
\end{proof}

\newpage

\section{Appendix}
\begin{center}
{\large by Francesc Fit\'e}
\footnote{Universit\"at Duisburg-Essen, Fakult\"at f\"ur Mathematik/ Institut f\"ur Experimentelle Mathematik, D-45127, Essen, Germany ; francesc.fite@gmail.com}
\end{center}

Note that Propositions \ref{prop: B[C_2]} and \ref{prop: E[D_n]} need the assumption that a certain quadratic extension $F^{\prime}$ of $F$ is such that $F^{\prime}$ is a totally real field.  

This is not always the case for the examples in \cite[Table 11]{FKRS}, and one may wonder about the existence of obstructions imposed by the arithmetic of certain Galois types to this hypothesis ever being satisfied.
 
The purpose of this section is to show that this is not the case. We will present explicit examples of abelian surfaces satisfying the hypothesis of Proposition \ref{prop: B[C_2]} and Proposition \ref{prop: E[D_n]} (in each of their cases), showing that their statements are indeed never vacuous.

Examples meeting the assumptions of Proposition \ref{prop: B[C_2]} can be constructed in the following way. Consider an elliptic curve $E/F^{\prime}$, where $F^{\prime}$ is a real quadratic field, whose $j$-invariant $j(E)$ lies in $F^{\prime} \setminus \mb{Q}$. If $\sigma$ denotes the nontrivial endomorphism of $F^{\prime}/\mb{Q}$, then the abelian surface given by descending $E\times E^{\sigma}$ to $\mb{Q}$ satisfies the hypothesis of Proposition \ref{prop: B[C_2]}, provided that $E$ and $E^{\sigma}$ are not $\overline{\mb{Q}}$-isogenous. An explicit example of this construction is given by the Jacobian of the following genus 2 curve  
\begin{equation}\label{equation: dr curve}
C\colon y^2 = x(x-1)(x^4 + 2x^3 - 6x + 1)\,.
\end{equation}
The curve $C$ has a non-hyperelliptic involution
$$
\alpha(x,y)=\left(\frac{x-1}{x+1}, \frac{-\sqrt{2}y}{(x+1)^3} \right)\,.
$$
Let $F'=\mb{Q}(\sqrt 2)$ and let $E$ denote the quotient curve $C/\langle \alpha \rangle$. The elliptic curve $E$ may be given by the affine equation
$$
E\colon y^2 = x^3 + \frac{19\sqrt 2 + 22}{2}x^2 - \frac{18\sqrt {2} +
23}{2}x + \frac{7\sqrt{2} + 10}{4}\,.
$$
It follows that ${\rm Jac}(C)$ is isogenous to $E\times E^{\sigma}$ over $F^{\prime}$. Since $E$ has $j$-invariant $j(E)=\frac{1}{5}(2142720\sqrt{2} + 3039232)$, it does not have complex multiplication. One may now apply \cite[Lemma 4.12]{FKRS} (choosing $p=7$, for example) to deduce that $E$ and $E^{\sigma}$ are not $\overline{\mb{Q}}$-isogenous. It follows that $L=F^{\prime}$ and that the Galois type of ${\rm Jac}(C)$ is \textbf{B}$[C_2]$. 

There are also examples of absolutely simple abelian surfaces of type \textbf{B}$[C_2]$ with $L$ totally real. As Benjamin Smith pointed out to us, one may find such examples inside the family
$$
C_t\colon y^2 = x^5 - 5x^3 + 5x + t\,,
$$
for some choices of $t\in \mb{Q}$. In \cite{TTV91} it is shown that ${\rm End}_{\overline{\mb{Q}}}({\rm Jac}(C_{t}))^{0}$ contains $\mb{Q}(\sqrt{5})$ with this real multiplication being only defined over $\mb{Q}(\sqrt{5})$. For particular choices of $t\in \mb{Q}$, one can argue that ${\rm Jac}(C_t)$ has no further endomorphisms. To do this, one can use the factorization of a few local factors to show that ${\rm Jac}(C_t)$ is absolutely simple, and that it does not have neither complex nor quaternionic multiplication. For example, the choice $t=1$ works, by looking at the local factors at $p=11$ and $p=19$.

Regarding Proposition \ref{prop: E[D_n]}, the case \textbf{E}$[D_2]$ is clear: since any biquadratic extension of $\mb{Q}$ contains at least one real quadratic extension, Proposition \ref{prop: E[D_n]} applies to any abelian surface over $\mb{Q}$ of type  \textbf{E}$[D_2]$ (for example, the one given in \cite[Table 11]{FKRS}).

For the case \textbf{E}$[D_4]$, we will use the parametrizations of genus 2 curves with automorphism group isomorphic to the dihedral group of $8$ elements $D_4$, given in \cite{CQ07}. The choice of parameters
$$
s=1,\qquad u=1/2,\qquad v=2,\qquad z=1\,,
$$
in \cite[Prop. 4.3]{CQ07} yields a curve isomorphic to
$$
C\colon y^2= 15x^6 - 48x^5 - 6x^4 - 12x^2 + 192x + 120\,.
$$
It is well-known that in this case one has ${\rm Jac}(C)\sim_{\tilde L} E^2$, where $E$ is an elliptic curve defined over the field $\tilde L$ of definition of the automorphisms of $C$. By \cite[\S2]{Car04}, the $j$-invariant of $E$ has two possibilities
$$
j(E)=\frac{2^ 6(3\pm 5\sqrt 2)^ 3}{(1\pm\sqrt 2)(1\mp \sqrt 2)^ 2}\,,
$$
from which one can deduce that $E$ does not have complex multiplication. This implies that $L=\tilde L$. We may now use \cite[Prop. 3.3]{CQ07} to deduce that
$$
L=\mb{Q}\left( \sqrt{2},\sqrt{1-1/\sqrt{3}} )\right)\,.
$$
The Galois type of ${\rm Jac}(C)$ is \textbf{E}$[D_4]$, provided that $Gal(L/\mb{Q})\simeq D_4$. To conclude, note that the three quadratic subextensions $\mb{Q}(\sqrt 3)$, $\mb{Q}(\sqrt 2)$, and $\mb{Q}(\sqrt 6)$ of $L/\mb{Q}$ are real.

Analogously, for the case \textbf{E}$[D_6]$, we will use the parametrizations of genus 2 curves with automorphism group isomorphic to the dihedral group of $12$ elements $D_6$, also given in \cite{CQ07}.
Choosing now parameters
$$
s=1,\qquad u=2,\qquad v=7/3,\qquad z=1\,,
$$
in \cite[Prop. 4.9]{CQ07}, we obtain a curve isomorphic to
\begin{equation}\label{eq: D6}
C\colon y^2 = x^5 + 12x^4 - 2124x^3 + 7992x^2 + 329184x - 38880\,.
\end{equation}
As in the previous example, one has ${\rm Jac}(C)\sim_{\tilde L} E^2$, where $E$ is an elliptic curve defined over the field $\tilde L$ of definition of the automorphisms of $C$. By \cite[\S2]{Car04}, the $j$-invariant of $E$ has two possibilities
$$
j(E)=\frac{2^8 3^3(2\pm 5\sqrt 2)^ 3(\mp\sqrt 2)}{(1\pm 2\sqrt 2)(1\mp 2\sqrt 2)^3}\,,
$$
from which one can deduce that $E$ does not have complex multiplication. Again this implies that 
$L=\tilde L$, which may be computed by means of \cite[Prop. 3.5]{CQ07}. Indeed, one finds that $\tilde L$ is the compositum of $\mb{Q}(\sqrt{2})$ and the splitting field of
$$
x^3-\frac{3}{2}x-\frac{1}{4}\,.
$$
Since $Gal(L/\mb{Q})\simeq D_6$, we have that the Galois type of ${\rm Jac}(C)$ is \textbf{E}$[D_6]$. The three quadratic subextensions $\mb{Q}(\sqrt{2})$, $\mb{Q}(\sqrt{42})$, and $\mb{Q}(\sqrt{21})$ of $L/\mb{Q}$ are all real. Taking $F^{\prime}=\mb{Q}(\sqrt {21})$, we see that ${\rm Jac}(C)/F^{\prime}$ has Galois type \textbf{E}$[C_6]$.

For the case \textbf{E}$[D_3]$, consider ${\rm Jac}(C)/F$, where $C$ is as in (\ref{eq: D6}) and $F=\mb{Q}(\sqrt{2})$. The Galois type of ${\rm Jac}(C)/F$ is \textbf{E}$[D_3]$ and the Galois type of ${\rm Jac}(C)/F^{\prime}$, with $F^{\prime}=\mb{Q}(\sqrt{2},\sqrt{21})$, is \textbf{E}$[C_3]$.

\textbf{Acknowledgements.} Thanks to Andrew Sutherland for assistance in several computations, as well as for kindly providing the curve of (\ref{equation: dr curve}), which was found using the techniques of \cite{KS08}.

\end{document}